\documentclass[11pt,english]{article}

\usepackage[T1]{fontenc}
\usepackage[utf8]{inputenc}

\usepackage{xcolor}

\usepackage{graphicx}
\usepackage{mathtools,amsfonts,amssymb}
\usepackage{mathrsfs} 
\usepackage{interval}

\definecolor{vertfonce}{rgb}{0.20, 0.46, 0.25}
\definecolor{rougefonce}{rgb}{0.64, 0.09, 0.20}

\usepackage{babel}
\usepackage[breaklinks=true,
            colorlinks=true,
            linkcolor=rougefonce,
            citecolor=vertfonce]{hyperref}

\newcommand{\RM}{\mathbb{R}}
\newcommand{\ZM}{\mathbb{Z}}

\newcommand{\NM}{\mathbb{N}}
\newcommand{\CM}{\mathbb{C}}
\newcommand{\ham}[1]{\mathcal{X}_{#1}}
\newcommand{\dd}[1]{\ensuremath{\operatorname{d}\!{#1}}}
\newcommand{\h}{\hbar}
\newcommand{\ohb}{{\O}(\hbar^\infty)}
\newcommand{\phy}{\varphi}
\newcommand{\abs}[1]{\left|#1\right|}
\newcommand{\norm}[1]{\left\|#1\right\|}
\newcommand{\deriv}[2]{\frac{\partial #1}{\partial #2}}
\newcommand{\restr}{\upharpoonright}
\newcommand{\ssi}{\Longleftrightarrow}

\newcommand{\theor}{Theorem}
\newcommand{\defin}{Definition}
\newcommand{\lemma}{Lemma}
\newcommand{\remar}{Remark}

\newcommand{\corol}{Corollary}
\newcommand{\propo}{Proposition}
\newcommand{\demon}{Proof}
\newcommand{\probl}{Problem}
\newtheorem{theo}{\theor}[section]
\newtheorem{theo*}{\theor}
\newtheorem{defi}[theo]{\defin}
\newtheorem{defi*}[theo*]{\defin}
\newtheorem{prop}[theo]{\propo}

\newtheorem{coro}[theo]{\corol}

\newcommand{\cqfd}{\hfill $\square$\par\vspace{1ex}}
\newcommand{\demons}[1][$\!\!$]{\noindent\textbf{\demon\ }\textsl{#1}\textbf{.}~}

\newenvironment{rema}
{\par\vspace{1ex}\refstepcounter{theo}%
\noindent\textbf{\remar~\thetheo} }
{~\hfill\mbox{$\triangle$}\par\vspace{1ex}}

\newenvironment{demo}[1][$\!\!$]
{\demons[#1]\ }
{\cqfd}

\numberwithin{equation}{section}

\DeclareUnicodeCharacter{1ECD}{\d o}

\newcommand{\C}{\mathscr{C}}
\renewcommand{\O}{\mathcal{O}}

\title{WKB for semiclassical operators:\\
  How to fly over caustics (and more)}

\author{\textsc{Vũ Ngọc} 
  San\footnote{Univ Rennes, CNRS, IRMAR --- UMR 6625, F-35000 Rennes, France}}

\begin{document}
\maketitle


\begin{abstract}
  The method initiated by Wentzel, Kramers, and Brillouin to find
  approximate solutions to the Schrödinger equation lies at the origin
  of the spectacular development of microlocal and semiclassical
  analysis. When used naively, the approach appears to break down at
  caustics, but Maslov showed how a simple generalization could
  overcome this difficulty. In this paper, after a partial historical
  review, we take advantage of more recent advances in microlocal
  analysis to present a unified treatment of this generalized
  Maslov-WKB method, using the microlocal sheaf-theoretic approach
  of~\cite{san-focus}. This framework provides a rigorous proof of the
  Bohr–Sommerfeld–Einstein–Brillouin–Keller quantization conditions
  for the eigenvalues of general semiclassical operators
  (pseudodifferential and Berezin–Toeplitz) in one degree of
  freedom. We also review some applications and extensions.
\end{abstract}

\section{EBK and WKB}

In 1926, almost simultaneously, three papers appeared to propose an
approximate solution to the newly published Schrödinger
equation~\cite{schrodinger26}:
\begin{equation}\label{equ:schrodinger}
  -\frac{\hbar^2}{2m} \frac{d^2 \Psi}{dx^2} + V \Psi = E \Psi,
\end{equation}
where $m$ is the particle mass, $V:x\mapsto V(x)$ the potential,
$\Psi:x\mapsto \Psi(x)$ the wavefunction, and $E$ the energy
eigenvalue. The three authors are Wentzel~\cite{Wentzel1926} (received
on June 18, 1926), Brillouin~\cite{Brillouin1926-cras,brillouin1926}
(received on July 1, 1926 and October 13, 1926),
Kramers~\cite{Kramers1926} (received on September 9, 1926).

This approximate solution, now called the WKB Ansatz, has the form
\begin{equation}
  \label{equ:wkb}
  \Psi(x) = a_\hbar(x) \exp\left(\frac{i}{\hbar} \varphi(x)\right),
\end{equation}
where the reduced Planck constant $\hbar$ is treated as a small real
parameter, and the function $a_\hbar(x)$ admits an asymptotic expansion in
non-negative integral powers of $\hbar$:
\[
  a_\h(x) = a_0(x) + \h a_1(x) + \h^2 a_2(x) + \cdots
\]
in some adequate topology. The fact that the small parameter appears
in front of the derivative in~\eqref{equ:schrodinger} make the limit
$\h\to 0$ \emph{singular}, in the sense that the ``kinetic'' term
$ -\frac{\hbar^2}{2m} \frac{d^2 \Psi}{dx^2} $ should \emph{not} be
considered smaller than the ``potential'' term $V\Psi$. The WKB
Ansatz~\eqref{equ:wkb} reflects this by the highly oscillatory
behavior of the phase (the function $\phy$ is real-valued.)

Both historically and scientifically, it would have been interesting
to be there, back in 1926, to interview these three scientists and
understand how the very same idea came up to their minds at the same
time. Investigating this nowadays would require the competence of a
historian, which I do not possess; the only thing I can say is that
Brillouin's papers mention Wentzel, and Kramers' paper, which, by the
way, was written in Utrecht\footnote{where Duistermaat, another name
  in this story, spent most of his career.}, does mention Brillouin
and Wentzel's work, which makes the usual ordering of the three
initials difficult to justify\footnote{The WKB ordering reflects the
  publications dates if one omits the CRAS
  note~\cite{Brillouin1926-cras}. French researchers often use BKW
  because it's alphabetical --- or for a less honorable reason.}. I
have no clue as to whether this was reciprocal.

On the other hand, this idea didn't come from nowhere, not at all. Already
in 1899, there were several papers by Horn~\cite{horn1899a,horn1899b}
treating the same differential equation with the same Ansatz, and he
himself based his work on previous studies by Sturm and others. Historical
papers usually also mention the work of Jeffreys~\cite{jeffreys1925}, which
interestingly was published just a few months before the famous Schrödinger
papers. So, it is actually difficult to claim that the WKB Ansatz was
really novel in 1926. On the contrary, one should not be surprised that, as
soon as Schrödinger published his equation, many researchers tried to find
a way to solve it and came up with this Ansatz.

What is really interesting in the WKB approach is that these authors were
the first --- to my knowledge --- to realize that the asymptotic expansion
of Horn and others actually enlightens the relationship between the old
classical mechanics of Hamilton and the new quantum mechanics of
Schrödinger. In particular, the phase $\varphi$ must satisfy the now called
Hamilton–Jacobi equation:
\begin{equation}
  H\left(x, \frac{\partial \varphi}{\partial x}\right) = E\,,
\end{equation}
where $H$ is Hamilton's function\footnote{From now on, I incorporate the
  mass $m$ into the new semiclassical parameter $\frac{\h}{m}$. }:
\[
  H(x,\xi) = \frac{1}{2}\norm{\xi}^2 + V(x)\,.
\]
Even more fascinating, they show how this Ansatz gives a very simple
justification of the EBK quantization rule of the old quantum theory.

This rule, a generalization of the Bohr-Sommerfeld quantization formula,
claims that in order for the energy $E$ to be a correct eigenvalue of the
Schrödinger operator, it should satisfy the following geometrical equation:
\begin{equation}
  \frac{1}{2\pi \hbar} \oint_{\gamma_E} \xi \, \dd{x} = n  \in \mathbb{Z},
\end{equation}
where $\gamma_E$ denotes a closed classical orbit in phase space with
energy $E$, and $\xi$ is the classical momentum (denoted by $p$ in the
physics literature) conjugate to the position coordinate $x$ (often
denoted by $q$). The integer $n$ is called the quantum number.

Like WKB, EBK is another concatenation of three authors' initials:
Einstein, Brillouin (the same Brillouin), and Keller. But this time,
contrary to WKB, the historical progression is clear. To make it
simple, one can say that Einstein~\cite{einstein-1917} proposed the
quantization rule based on geometric considerations (and of course on
earlier works, in particular by Bohr, Ehrenfest, Epstein,
Schwarzschild, Sommerfeld), then Brillouin showed the connection with
Schrödinger's new quantum mechanics (this is the same paper as the one
cited for WKB!), and finally Keller~\cite{keller-bs} explained that,
in order to be more accurate, the rules require a correction (now also
called the Maslov correction):
\begin{equation} \frac{1}{2\pi \hbar} \oint_{\gamma_E} \xi \, dx = n +
  \frac{\mu}{4}, \quad n \in \mathbb{Z},
\end{equation}
where $\mu$ is called the (Keller-)Maslov index; it is defined from
topological properties of the loop $\gamma_E$ and accounts for phase
shifts at caustics or turning points. This remarkably simple rule,
(partly) justified by the WKB Ansatz, bridges classical mechanics and
quantum mechanics by imposing discrete energy levels based on
classical trajectories and corrected by topological information.

\section{From WKB to Microlocal Analysis}

This is not the end of the story. In fact, I would rather say that the
WKB Ansatz was the starting point of a new and fascinating
mathematical development of intimidating breadth: microlocal analysis.

Indeed, mathematicians had decades of headaches trying to properly
justify the WKB approach. The first problem comes from turning points,
or caustics. These are points where the approximation ceases to be
uniform and in fact blows up. For the one-dimensional Schrödinger
equation, these turning points correspond, on the Hamiltonian side, to
a position $x$ where the momentum $\xi$ of the particle
vanishes. Thus, the energy level set has there a vertical tangent,
and hence the phase function $\varphi$ has an infinite derivative,
which forbids the Hamilton–Jacobi equation to be solved by smooth
functions.

At this point, it is difficult not to notice that the EBK quantization
rule, by contrast, has no singular behavior whatsoever at turning
points. So is EBK better than WKB, or more general? This is really
tempting to say so, because in fact not only does EBK have no
singularity, but it can also be generalized to any kind of Hamiltonian
function, not specifically the one obtained from Schrödinger
operators, which indicates that it could be used to solve eigenvalue
problems arising from more general partial differential equations.
The extension of the WKB method to a large class of operators was done
thirty years later by Lax~\cite{lax57}, but he was still constrained
by the caustic issue. It seems that Maslov~\cite{maslov} was the first
to substantially extend the WKB Ansatz to a form that would swallow
caustics using partial Fourier transforms. Then,
Hörmander~\cite{FIO1,hormander-all}, also inspired by Sato's
microfunctions~\cite{sato70}, wrapped up the whole thing into a
formidable theory now called (homogeneous) microlocal analysis.
Unfortunately, Hörmander neither wanted to motivate his theory by
quantum mechanics nor to apply it to this field. He was interested in
the regularity theory of distributional solutions to general partial
differential equations and did not use any small parameter
$\hbar$. Duistermaat~\cite{duistermaat-oscillatory} was the first to
incorporate the semi-classical parameter $\hbar$ back again in order
to use microlocal analysis towards recovering Maslov's results. This
was the birth of semiclassical analysis. In doing so, Duistermaat
relied on the theory of Fourier integral operators, recently developed
by Hörmander and himself~\cite{FIO2} (following a thread of ideas
including in particular the work of Egorov~\cite{egorov}), in order to
generalize the notion of WKB solutions, now called semi-classical
lagrangian distributions. Formally (that is, ignoring the problem of
convergence of integrals), a lagrangian distribution is simply a
linear superposition of standard WKB Ansätze, as follows:
\begin{equation}\label{equ:lagrangian}
  \Psi(x) = \frac{1}{(2\pi\hbar)^{n/2}}
  \int_{\RM^n} e^{\frac{i}{\hbar} \varphi(x, \theta)} \, a_{\h}(x, \theta) \, \dd{\theta}\,.
\end{equation}

\section{The lagrangian manifold viewpoint}

Consider the phase space $\mathbb{R}^{2n}$ equipped with the canonical
symplectic form
$\omega = \dd{\xi} \wedge \dd x = \sum_{j=1}^n \dd{\xi}_j \wedge \dd
x_j$. A lagrangian\footnote{I follow here Weinstein's
  use~\cite{weinstein-symplectic} of the fully adjectival form
  ``lagrangian''--- that is, without capital ``L''; I also second
  Weinstein in ending the word with ``ian'' instead of ``ean''; after
  all, Joseph Louis de Lagrange was born Giuseppe Luigi Lagrangia.}
(sub)manifold is a submanifold of $\mathbb{R}^{2n}$ of dimension $n$
on which the symplectic form vanishes. For instance, the ``horizontal
fiber'' $\{\xi=\textup{const}\}$ is a lagrangian submanifold (of some
importance, as we shall see later). Of course, the same notion extends
to any symplectic manifold. What do lagrangian manifolds have to do
with the Schrödinger equation? Well, Alan Weinstein had this now
famous sentence: ``everything is a lagrangian
manifold.''~\cite{weinstein-category}. I'm not going to fully justify
this, but let me explain why any WKB solution --- even in the
generalized Maslov sense~\eqref{equ:lagrangian} --- of the Schrödinger
equation is associated to a lagrangian manifold.

Recall that the phase $\varphi$ is a solution to the Hamilton-Jacobi
equation \begin{equation} H\left(x, \deriv{\varphi}{x}\right) =
  E.  \end{equation} So we have a natural subset of the energy level
set, defined by the set $\Lambda_\phy$ of all pairs
$(x, \deriv{\varphi}x)$, which is the graph of the differential of
$\varphi$.  I claim that this graph is a lagrangian
submanifold. Indeed, if you restrict the symplectic form
$\dd{\xi} \wedge \dd x$ to this graph, you
obtain: \begin{equation} \label{equ:phase} \dd{\xi} \wedge \dd
  x_{\restr \Lambda_\phy} = \dd(\xi\dd{x})_{\restr \Lambda_\phy} =
  \dd(\dd\varphi) = 0.  \end{equation} (The Liouville 1-form
$\alpha=\xi\dd x$ has this remarkable ``tautological'' property that,
when restricted to a section of the cotangent bundle --- \emph{i.e.},
a 1-form ---, it produces back the 1-form itself, here $\dd \phy$.)

In the case of the generalized WKB Ansatz, as in
Equation~(\ref{equ:lagrangian}), and under generic conditions on
$\varphi$, we can see by the stationary phase lemma that the integral
localizes on the critical points of the map \[ \theta \mapsto
\varphi(x, \theta), \] that is, on the set \[ C_\varphi = \left\{ (x,
\theta) \ \middle| \ \frac{\partial \varphi}{\partial \theta}(x,
\theta) = 0 \right\}.  \] At these localized points, the asymptotic
expansion of $\Psi(x)$ becomes a usual WKB Ansatz. Hence, one can show
that if it is a solution to the Schrödinger equation, we must have the
Hamilton-Jacobi equation \[ H\left(x, \frac{\partial \varphi}{\partial
x}(x, \theta)\right) = E, \quad \text{where } (x, \theta) \in
C_\varphi.  \] And again, one can check that the image of $C_\varphi$
under the map \[ (x, \theta) \mapsto \left(x, \frac{\partial
\varphi}{\partial x}(x, \theta)\right) \] is a lagrangian
submanifold. Remarkably, one can prove conversely that a (non-empty)
class of lagrangian distributions which are locally of the
form~\eqref{equ:lagrangian} can be associated to \emph{any} lagrangian
submanifold; this is one the main achievements of Maslov, extended to
a great generality by Hörmander (but for homogeneous lagrangian
submanifolds) and then Duistermaat; for a thorough explanation, I
recommend Duistermaat's paper~\cite{duistermaat-oscillatory}. It
follows that, in order to find a good phase for a (generalized) WKB
solution to the Schrödinger equation (or to any good semiclassical
equation, see below), it is simply enough to find a arbitrary
lagrangian submanifold contained in the required energy level set
$H(x,\xi)=E$.

One of the (many) magical aspects of lagrangian manifolds is that such
a lagrangian is automatically \emph{invariant} by the Hamiltonian flow
of $H$. Indeed, if $v$ is tangent to $\Lambda$, then \[
\omega(\ham{H},v) = -\dd H \cdot v = 0\,, \] the first equality
defining the Hamiltonian vector field $\ham{H}$ and the second one
expressing that $v$ is tangent to the level set of $H$. Therefore,
$\ham{H}$ is symplectically orthogonal to $\Lambda$. But since
$\Lambda$ is lagrangian, its symplectic orthogonal is itself, and
hence $\ham{H}$ is tangent to $\Lambda$!  \footnote{More details
concerning the geometry of the WKB approximation can be found in the
nice book~\cite{weinstein-bates}.}

We now see that solving the Schrödinger equation by means of the WKB
Ansatz, at first order $\O(\h)$, amounts to finding an \emph{invariant
lagrangian submanifold} of the phase space. (By invariant we shall
always mean invariant under the Hamiltonian flow of the classical
Hamiltonian $H$.)

The reader may wonder why we introduce general lagrangian manifolds
here, while in this paper we are mostly interested in the
one-dimensional case: the phase space is $\mathbb{R}^2$, or a
two-dimensional surface, and a lagrangian manifold is just a curve in
that space.  Well, one reason is that the first part of the WKB
papers, which deals with approximate solutions to the Schrödinger
equation, actually works in any dimension and hence defines general
lagrangian submanifolds, which provides a nice geometric
interpretation of the wave function $\Psi$. However, the second part,
which aims at justifying the Bohr–Sommerfeld quantization condition,
is much more delicate and only works in dimension 1, or for separable
or completely integrable systems. But, and this is our second reason,
it turns out that our main tool for a rigorous justification of the
Bohr-Sommerfeld condition, even in dimension 1, will be Fourier
Integral Operators, which themselves are actually WKB functions
defined in a higher dimensional space.

\section{Semiclassical operators and microlocal solutions}

Microlocal analysis was first developed for (semiclassical)
pseudodifferential operators, which is a large class of linear
operators containing all differential operators with smooth
coefficients (and hence the Schrödinger
operator)\footnote{Technically, one requires a moderate growth of the
coefficients; for a Schrödinger operator, this means that $V$ must be
smooth and behave at most polynomially at infinity. In current
microlocal terminology, the symbols of our operators must belong to an
appropriate \emph{class}, which we shall implicitly assume throughout
this article. See for instance~\cite{zworski-book-12} for
details.}. It was then realized that the whole theory can also be
applied to a very different class of operators, the so-called
Berezin-Toeplitz operators (see for
instance~\cite{lefloch-book}). These are not differential
operators. Instead, for each fixed value of $\hbar$, they act on a
finite-dimensional Hilbert space, the dimension of which grows to
infinity as $\hbar \to 0$. In this paper, following an idea
of~\cite[Appendix A]{san-yohann21}, we will call \emph{semiclassical
operators} operators that are either semiclassical pseudodifferential
operators or Berezin-Toeplitz operators. In both cases, we shall
denote by $\mathcal{H}$ the natural Hilbert space on which they act
(this space depends on $\h$ in the Berezin-Toeplitz case), and by $M$
the phase space (so $\mathcal{H}=L^2(\RM^n)$ and $M=\RM^{2n}$ in the
pseudodifferential case).

In the Berezin-Toeplitz case, the admissible values of $\h$ are
quantized: typically, $\h=1/k$ with $k\in\NM^*$. However, I wrote this
text with a bias toward pseudodifferential operators; hence, for
notational simplicity, we shall say ``$\h \in \interval[open
left]01$'' when we actually mean ``$\h$ admissible and inside
$\interval[open left]01$''. We shall also denote by $(x,\xi)$ a point
in phase space, even though for a general symplectic manifold there
might not be such global coordinates. In the Berezin-Toeplitz
quantization scheme, the WKB phase $e^{\frac{i}{\h}\phy(x)}$ must be
replaced by a holomorphic section of a prequantum line bundle, which
is flat on some lagrangian manifold $\Lambda$,
see~\cite{charles-bs}. (This also covers the ``generalized WKB
Ansatz''~\eqref{equ:lagrangian}.)

Let $P$ be a semiclassical operator with symbol $H(x, \xi)$. We say
that $\psi\in\mathcal{H}$ is a WKB solution of order $N$ to the
equation $(P - E)\psi = 0$ if $\psi$ is a lagrangian distribution
whose associated lagrangian submanifold is contained in the energy
level set $H = E$, and such that $\psi$ is a quasimode of order $N$,
that is \begin{equation} (P - E)\psi =
\mathcal{O}(\hbar^N).  \label{eq:quasimode} \end{equation} For
instance, when $P=\frac{\h}{i}\deriv{}{x}$, which is arguably the
simplest semiclassical operator, then $\psi(x) = C
e^{\frac{i}{\h}Ex}$, $x\in\RM$, is a \emph{plane wave} and a
lagrangian distribution associated with the lagrangian manifold
$\{(x,\xi); \quad \xi=E\}$, and is a WKB solution
to~\eqref{eq:quasimode} at any order.

For a lagrangian distribution, the quasimode
condition~\eqref{eq:quasimode} with $k = 1$ is automatic. However, for
larger $k$, one has to solve transport equations beyond the
Hamilton-Jacobi equation in order to obtain higher-order solutions. In
this paper, I do not describe the transport equations, because they
will be hidden in the more general theory of semiclassical Fourier
integral operators.

In which sense, or in which topology, do we want
Equation~\eqref{eq:quasimode} to hold?  It will be convenient to use
the notion of a \emph{microlocal solution}, which asserts that
Equation~\eqref{eq:quasimode} holds microlocally near a point $(x,
\xi)$ in phase space if and only if it holds in the $\mathcal{H}$ norm
after microlocal projection in a neighborhood of that point.  And, by
microlocal projection, we mean a semiclassical operator whose symbol
is $1$ in a small neighborhood of $(x, \xi)$ and $0$ outside some
compact set. An element $u_\h\in\mathcal{H}$ is said to be
\emph{microlocally supported}\footnote{or even ``microsupported''} in
some set $\Sigma\subset M$ if it is microlocally $\ohb$ outside of
$\Sigma$. The microlocal support is also called the semiclassical
wavefront set. See for instance~\cite[Definitions A.1 and
A.4]{san-yohann21}

Using semiclassical partitions of unity, one can show that the set of
microlocal solutions has a nice sheaf structure. This means that if
you have two solutions defined on two different small open sets with a
non-empty intersection, and which coincide on the intersection, then
they can be glued to produce a solution on the union of the open
sets. As we shall see, one can always choose the second solution such
that the coincidence on this intersection holds. In other words, any
\emph{local section} of the sheaf over the first open set extends to a
section over the union of both open sets. However, even if we have a
family of open sets covering a compact invariant lagrangian
submanifold, if that manifold is not simply connected, a \emph{global}
section of the sheaf might not always exist. If it does exist (which
in general will impose some conditions on $E$), then the solution
actually extends to a true global solution in $L^2$ of
Equation~\eqref{eq:quasimode}. Therefore, this proves (at least for
self-adjoint operators) that the energy $E$ is $\h^k$-close to the
spectrum of the operator $P$. These ideas were implemented by Colin de
Verdière in the pioneer article~\cite{colin-quasi-modes}.

Let us now see how it works more precisely in the one-dimensional
case.

\section{The Bohr-Sommerfeld cocycle}

Let $P$ be a self-adjoint semiclassical operator with discrete
spectrum in a spectral window $[E_1,E_2]$ (for instance, the
Schrödinger operator with a confining potential, or any
Berezin-Toeplitz operator on a compact Kähler manifold).

We shall assume from now on that $P$ has only one degree of freedom:
the phase space has dimension 2, and that the interval
$\interval{E_1}{E_2}$ is contained in the set of regular values of the
principal symbol $H$ of $P$. Moreover, we assume that
$H^{-1}(\interval{E_1}{E_2})$ is compact. We wish here not only to use
the WKB Ansatz to obtain approximate solutions to $(P-E)\psi = 0$, but
actually to show that \emph{all} solutions are indeed obtained in this
way, and in turn to obtain a very precise description of the spectrum
of $P$ in $[E_1,E_2]$.

For Schrödinger operators, specific proofs can be found in the
literature, see for instance the interesting lecture by
Voros~\cite{voros-bkw}. But for a statement of this generality, a full
microlocal approach is more appropriate. To my knowledge, the
papers~\cite{rozenbljum78,colinII,helffer-robert2} were among the
first ones to employ deep microlocal analysis to derive the
Bohr-Sommerfeld rules. Here, my plan is to be ``even more
microlocal'', in the sense that we start to describe the problem
locally near a point in phase space, and try to pass from local to
global. The strategy follows~\cite{san-focus,san-panoramas}, but I
shall give more details here.

Given an arbitrary function $\h\mapsto E(\h)$, let us denote by
$\mathcal{D}=\mathcal{D}_{(P,E)}$ the sheaf of microlocal solutions of
$P-E(\h)$. In other words, the set $\mathcal{D}(\Omega)$ of sections
of $\mathcal{D}$ over an open subset $\Omega\subset M$ is the space of
microlocal solutions $\psi_\h$ to the equation $(P-E(\h))\psi_\h=\ohb$
on $\Omega$.

First note that, by standard semiclassical ellipticity, if $E(\h)\to
E_0$, then $\mathcal{D}$ is supported in the energy level set $H=E_0$.
Since the phase space has dimension 2, this level set is generically
just a curve and hence a lagrangian submanifold by itself.  The key
result, which says that $\mathcal{D}$ is locally given by WKB
solutions near any regular point of $H$, is the following.

\begin{prop}[\cite{san-focus}]\label{prop:dim1} Near any regular point
$m$ of $H$, the space of microlocal solutions to $(P-E)\psi = 0$ is
one-dimensional, more precisely: \begin{itemize} \item There exists a
fixed\footnote{\emph{i.e.} independent of $\h$} neighborhood $\Omega$
of $m$, and for any $E\in H(\Omega)$, a WKB solution (\emph{i.e.} a
lagrangian distribution) $\Psi_\h=\Psi_{\h,E}$, depending smoothly on
$E$, with $\norm{\Psi_\h}=1$, such that \[ (P-E)\Psi_\h = \ohb\text{
on } \Omega\,.  \] \item If $\psi=\psi_\h$ satisfies $(P-E)\psi_\h =
\ohb$ near $m$ and $\norm{\psi_\h}=1$, then there exists a WKB
solution $\Psi_\h$ and a constant phase $C=C(E,\h)\in U(1)$, such
that \[ \psi_\h - C(\h)\Psi_\h = \ohb \text{ on }
\Omega\,.  \] \end{itemize} In both items the remainder $\ohb$ is
locally uniform in $E$.  \end{prop}

When we say that a WKB solution depends smoothly on $E$ we mean that,
in~\eqref{equ:lagrangian}, both the phase $\phy$ and the amplitude
$a_\h$ depend smoothly on $E$, uniformly with respect to $\h$.  This
statement can be seen as a semiclassical variant of~\cite{FIO2}; it
relies on the theory of Fourier integral operators (FIO).

FIOs are perhaps the most powerful tool in microlocal analysis. One of
their celebrated uses is to quantize canonical transformations of
phase space as unitary operators. If $U$ is such an FIO quantizing the
canonical transformation (\emph{i.e.}  symplectic diffeomorphism)
$\chi$, then for any semiclassical operator $P$ with symbol
$H(x,\xi)$, the conjugated operator \[ U^{-1}P U \] is again a
semiclassical operator, and its symbol is $H\circ \chi + \O(\h)$ (this
statement is called the Egorov Theorem).  See for
instance~\cite[Proposition 4.3]{charles-bs} or~\cite[Theorem
11.5]{zworski-book-12}.

Remarkably, we have here another illustration of Weinstein's
``symplectic creed''~\cite{weinstein-category}, because FIOs are also
lagrangian distributions! (More precisely, the graph of $\chi$ in the
product space $\RM^{2n} \times \RM^{2n}$ is automatically a lagrangian
submanifold\footnote{this holds if the product $\RM^{2n} \times
\RM^{2n}$, sometimes denoted by $\RM^{2n} \times \overline{\RM^{2n}}$,
is equipped with the direct sum symplectic form $\omega \oplus
(-\omega)$}, and by definition an FIO is an operator whose Schwartz
kernel is a lagrangian distribution associated with this graph.)

  \begin{demo}[of Proposition~\ref{prop:dim1}] Since $m$ is regular:
  $\dd{H}(m) \neq 0$, there exist local canonical coordinates
  $(x,\xi)$ near $m$ such that $H(x,\xi) - H(m) = \xi$ when
  $(x,\xi)\in\RM^2$ stay in some neighborhood $\tilde \Omega$ of the
  origin (This the statement of Darboux-Carathéodory's theorem.)  Let
  $U$ be an FIO associated with this change of variables. Let
  $E_0=H(m)$. We obtain, microlocally near $m$, \[ U^{-1} (P - E_0) U
  = \frac{\h}{i}\deriv{}{x} + \O(\h)\,.  \] By a standard procedure
  (\emph{this is where the WKB transport equations are hidden!}
  See~\cite[Section 3.2]{san-fn} or the foundational
  paper~\cite[Proposition 6.1.4]{FIO2} --- in the homogeneous case.)
  one can multiply $U$ by a suitable unitary semiclassical operator
  $V$ such that, with $\tilde U := UV$, \[ {\tilde U}^{-1} (P - E_0)
  \tilde U = \frac{\h}{i}\deriv{}{x} + \ohb\,.  \]

    Working in coordinates, one can show that
  Proposition~\ref{prop:dim1} holds for the operator
  $\frac{\h}{i}\deriv{}{x} - E$; indeed, its solutions are plane waves
  $\psi(x) = C e^{\frac{i}{\h}Ex}$. This proves the proposition,
  with \[ \Psi_\h = C \tilde U e^{\frac{i}{\h}(E-E_0)x}\,.  \] Since
  $E\in H(\Omega)$, there exists $(x_0,\xi_0)\in\tilde\Omega$ such
  that $\xi_0=E-E_0$; hence the microsupport of
  $e^{\frac{i}{\h}(E-E_0)x}$ intersects $\tilde\Omega$, and we can
  normalize it by a uniform constant. The fact that, for any fixed
  $E$, such a $\Psi_\h$ is indeed a generalized WKB Ansatz associated
  with the lagrangian manifold $\{H=E\}$ follows from the fact that
  the Schwartz kernel of $U$ is a lagrangian distribution. By a
  stationary phase argument, one shows that the application of the
  semiclassical operator $V$ does not modify this (this is called the
  semiclassical FIO calculus, see for instance~\cite[section
  9.7]{guillemin-sternberg-semiclassical} or~\cite[Proposition
  2.7]{charles-bs}).  \end{demo}

  \begin{rema}\label{rema:parameter} In Proposition~\ref{prop:dim1},
  the energy parameter $E$ is purely decorative. Indeed, if $P$
  depends on an additional parameter $\lambda$, then the parametric
  version of the proposition holds. Therefore, one can set $P_E :=
  P-E$ and describe solutions to $P_E\Psi_\h = \ohb$ by viewing $E$ as
  a parameter. This remarks becomes far-reaching for spectral problems
  where the ``spectral'' parameter does not appear in a linear
  fashion.  \end{rema}

  This proposition allows us to understand in a simple way the sheaf
  $\mathcal{D}$ of microlocal solutions to the equation
  $(P-E)\psi=\ohb$, when the energy $E$ is a regular value of $H$:
  this sheaf is then simply a \emph{flat bundle}, meaning that for
  each point $m=(x_0,\xi_0)$ of the level set $H(x,\xi)=E$, there
  exists a neighborhood $U$ of $m$ in which any two microlocal
  solutions must differ by a multiplicative constant.

  Thus, at least theoretically, it is easy to determine when the sheaf
  has a \emph{global} section (\emph{i.e.} a microlocal solution in a
  full neighborhood of the energy level set $H=E$): this happens when
  the flat bundle is trivial, or equivalently when its first \v Cech
  cohomology group $\check H^1(\mathcal{D})$ vanishes (modulo $\ohb$).

  How to compute this cohomology? When $E$ is a regular value of $H$,
  we know that the level set $H=E$ is a finite union of closed smooth
  (embedded) curves $\C_k = \C_k(E)$, $k=1,\dots,d$. Let us consider a
  fixed $j$ and denote by $\mathcal{D}_k$ the restriction of
  $\mathcal{D}$ to a small tubular neighborhood of $\C_k$.  The curve
  $\C_k$ is diffeomorphic to a circle, and has a natural orientation
  given by the Hamiltonian vector field $\ham{H}$, which does not
  vanish on it by hypothesis. Let $\Omega_j$, $j=1,\dots,N_k$ be an
  open cover of $\C_k$ by small balls, on each of which
  Proposition~\ref{prop:dim1} hold, and whose centers $z_j\in \C_k$
  are ordered along the orientation of $\C_k$. For each $j$, the
  proposition gives a WKB function $\Psi_\h^{j}$ generating
  $\mathcal{D}(\Omega_j)$. By restricting to $\Omega_j\cap
  \Omega_{j+1}$ (indices are written modulo $N_k$), the uniqueness
  part of the proposition yields a constant phase $C_j = C_j(E,\h)\in
  U(1)$ such that \begin{equation} \Psi_\h^{(j+1)} = C_j \Psi_\h^{(j)}
  + \ohb \text{ on } \Omega_j\cap
  \Omega_{j+1}\,.  \label{equ:cocycle} \end{equation} \begin{defi} The
  assignment $\Omega_j\cap \Omega_{j+1} \mapsto C_j$ is called the
  \emph{Bohr-Sommerfeld cocycle} of the curve $\C_k$.  \end{defi}

  (It is indeed a cocycle in the sense of \v Cech cohomology.)

\begin{prop}\label{prop:bs} There exists a microlocal solution of
$(P-E)$ on a neighborhood of $\C_k$ if and only if the pair $(E,\h)$
satisfies the equation \begin{equation} \label{equ:bs} C_1(E,\h)\cdots
C_{N_k}(E,\h) = 1 + \ohb.  \end{equation} \end{prop} \begin{demo} The
solution exists if and only if the Bohr-Sommerfeld cocycle is a
coboundary, \emph{i.e.} there exist constants $d_j=d_j(E,\h)\in
\CM^*$, $j=1,\dots N_k$, such
that \begin{equation} \label{equ:coboundary} C_j = d_{j+1}^{-1} d_j +
\ohb \,.  \end{equation} Indeed, in this case, we see
from~\eqref{equ:cocycle} that the local sections $d_j \Psi_\h^{(j)}$
actually coincide (modulo $\ohb$) on the intersections $\Omega_j\cap
\Omega_{j+1}$. Therefore, they can be glued together by a microlocal
partition of unity to form a solution on a full neighborhood of
$\C_k$.

  Conversely, if such a global section $\Psi_\h$ exists, by
  Proposition~\ref{prop:dim1} we have constants $d_j$ such that
  $\Psi_\h = d_j \Psi_\h^{(j)} + \ohb$ on each $\Omega_j$. In view
  of~\eqref{equ:cocycle}, it follows that~\eqref{equ:coboundary}
  holds.

  It remains to show why the coboundary
equation~\eqref{equ:coboundary} is equivalent to~\eqref{equ:bs}. Since
indices $j$ are taken modulo $N_k$, it is clear
that~\eqref{equ:coboundary} implies~\eqref{equ:bs}. Conversely,
if~\eqref{equ:bs} holds, we let $d_1=1$, $d_j=(C_1\cdots
C_{j-1})^{-1}$ for $j=2,\dots, C_k$, which gives \[ d_1^{-1} d_{N_k} =
(C_1\cdots C_{N_k-1})^{-1} = C_{N_k} + \ohb, \] which
proves~\eqref{equ:coboundary} for $j=N_k$ (other $j$'s are
automatic.)  \end{demo}

Equation~\eqref{equ:bs} will be called the Bohr-Sommerfeld condition
of the curve $\C_k$. Since the curves for different $k$ are disjoint,
we obtain: \begin{prop} There exists a non-trivial section of
$\mathcal{D}$, \emph{i.e.} a microlocal solution on a neighborhood of
$H=E$, if and only if \emph{at least one} of the Bohr-Sommerfeld
conditions associated with the curves $\C_1,\dots, \C_d$ is
satisfied.  \end{prop}

At this point, it is not clear how to relate this Bohr-Sommerfeld
condition to the EBK prediction. This is given by the following
proposition.  \begin{prop}\label{prop:semiclassical-action} The
semiclassical action $A_k(E,\h) := \log (C_1\cdots C_{N_k})$ admits an
asymptotic expansion of the form \begin{equation} \label{equ:A-das}
A_k(E,\h) = \frac{1}{\h} A_{k,0}(E) + A_{k,1}(E) + \h A_{k,2} + \h^2
A_{k,3} + \cdots \end{equation} with smooth coefficients $A_{k,j}$;
and \begin{equation} \label{equ:action} A_{k,0} = \int_{\C_k}
\alpha\,, \end{equation} where $\alpha$ is the Liouville
1-form.  \end{prop} \begin{demo} We use the same quantum
Darboux-Carathéodory normal form as in the proof of
Proposition~\ref{prop:dim1}, this time repeated on each
$\Omega_j$. Thus, we may obtain the Bohr-Sommerfeld cocycle $(C_j)$ by
choosing $\Psi_\h^{(j)} := \tilde U_j e^{\frac{i}{\h}(E-H(m))x}$. (If
we viewed $E$ as a parameter, see Remark~\ref{rema:parameter}, we
could simply take $\Psi_\h^{(j)} := \tilde U_j 1 $.)

  From~\eqref{equ:cocycle} we obtain, microlocally on $ \Omega_j\cap
  \Omega_{j+1}$, \[ C_j = e^{-\frac{i}{\h}(E-H(m))x} U_j^{-1} U_{j+1}
  e^{\frac{i}{\h}(E-H(m))x} + \ohb\,.  \] Since $C_j$ is constant, it
  can be evaluated on any point in $ \Omega_j\cap \Omega_{j+1}$. Since
  the Schwartz kernel of $U_j$ and $U_{j+1}$ are lagrangian
  distributions, \emph{i.e.} take the form~\eqref{equ:lagrangian},
  with an amplitude $a_\h$ admitting an asymptotic expansion in
  integral powers of $\h$, we obtain that, for each fixed $E$, $C_j$
  must have an asymptotic expansion of the
  form \begin{equation} \label{equ:Cj-das} C_j \sim
  e^{\frac{i}{\h}S_0}(a_0+\h a_1 + \h^2 a_2 +
  \cdots)\,, \end{equation} and all coefficients depend smoothly on
  $E$. Since $\abs{C_j}=1 +\ohb$ we must have $\abs{a_0}=1$. Hence we
  may also write \[ C_j \sim e^{\frac{i}{\h} \sum_{j\geq 0} \h^j
  S_j}\,, \] (as Brillouin~\cite{brillouin1926} suggested) which
  gives~\eqref{equ:A-das}.

  In order to compute $A_{k,0}$, recall that $\Psi_\h^{(j)}$ is a WKB
solution associated to the curve $\C_k\subset \{H=E\}$. Hence, near
any point $m$ in the lagrangian manifold $\C_k$, we have a phase
function $\phy_j$ as in~\eqref{equ:lagrangian};
using~\eqref{equ:phase}, this phase, viewed as a function on $\C_k$,
satisfies $\dd\phy_j = \alpha$, where $\alpha= \xi\dd x$ is the
Liouville 1-form.  Evaluating~\eqref{equ:cocycle} on a point $m$ of $
\Omega_j\cap \Omega_{j+1}$, we deduce that the coefficient $S_0
=S_0^{(j)}$ of~\eqref{equ:Cj-das} takes the form \[ S_0^{(j)} =
\phy_{j+1}(m) - \phy_j(m) \] but since $\dd (\phy_{j+1} - \phy_j)=0$,
the function $\phy_{j+1} - \phy_j$ is constant, and the set of all of
these constants, \emph{i.e.} $(S_0^{(j)})_{j=1,\dots, N_k}$ is exactly
the \v Cech cocycle associated with the closed 1-form $\alpha$ on
$\C_k$. (See for instance~\cite[Appendix C]{weinstein-bates}.) Thus,
$\sum_{j=1}^{N_k} S_0^{(j)} = \int_{\C_k}\alpha$, which
gives~\eqref{equ:action}.  \end{demo}

Actually, it is also possible to describe the second term $A_{k,1}$;
in the pseudodifferential case, if $P$ has no subprincipal symbol,
$A_{k,1}=\mu\frac{\pi}{2}$ where $\mu$ is the (Keller-Leray-)Maslov
index of $\C_k$ which, in $\RM^2$, is always equal to $2$. In the
Berezin-Toeplitz case, see~\cite{charles-bs}. In the special case of
the Schrödinger operator, an iterative scheme to obtain $A_{k,j}$ for
all $j$ is known~\cite{colin-bs, argyres65}.

\section{Quasi-modes and eigenvalues} Let us now upgrade the
microlocal result of Proposition~\ref{prop:dim1} to a concrete
statement about honest eigenvalues and eigenfunctions, in order to
fully justify the EBK rule.

First of all, as expected, the generalized WKB construction will give
us good quasimodes, \emph{i.e.} solutions $(\Psi_\h, E(\h))$ to \[
\norm{(P-E(\h))\Psi_\h}_{\mathcal{H}} = \ohb\,.  \] But we have much
more: exact eigenfunctions must be $\ohb$-close to these quasimodes,
and hence the whole spectrum in $\interval{E_1}{E_2}$ is $\ohb$-close
(in every possible sense, including multiplicity) to the set of $E$'s
for which the microlocal WKB construction was valid.

Most treatments of the Bohr-Sommerfeld rule in the literature make the
simplifying assumption that energy level sets are connected\footnote{A
notable exception is the recent paper~\cite{deleporte-lefloch25}.};
this is quite understandable, since the presentation becomes much
simpler, and adding connected components is ``well-known to the
experts''. From the point of view of symplectic geometry, this makes
no difference indeed, since each component can be treated
separately\footnote{Actually, there is a topological subtlety, which
is never discussed to my knowledge, concerning a distinction between
the global fibration by $H$ and the semiglobal one --- \emph{i.e.}  in
a neighborhood of a connected component: if $H^{-1}(E)$ is assumed to
be compact, then it is a union of circles, and near any circle $\C$,
the action-angle theorem asserts in particular that neighboring fibers
are also compact and connected\ldots which means: \emph{in a saturated
neighborhood of $\C$}; but this does \emph{not} imply that neighboring
global fibers of $H:M\to \RM$ are compact.}. The microlocal treatment
is equally similar. However, proving that WKB solutions exhaust the
spectrum (which is the goal of this section) requires an additional
argument, which we present here. One way to achieve this (and again,
this is known to experts, I don't claim much originality) is to use
the fact that quasimodes associated to different components are
mutually orthogonal modulo $\ohb$~\cite[Proposition 2.6]{charles-bs}.

One of the first mathematical accounts of the Bohr-Sommerfeld rule is
due to Helffer and Robert~\cite{helffer-robert2}; interestingly,
motivated by quantum tunneling, they did consider multiple components,
but with a symmetry assumption implying that all periods of the
Hamiltonian flow on $H^{-1}(E)$ are equal, allowing them to apply an
idea of Colin de Verdière~\cite{colin-bica}.

\begin{theo}[Bohr-Sommerfeld eigenvalues]\label{theo:bs} Let $P$ be a
self-adjoint semiclassical operator with one degree of freedom and
spectrum $\sigma(P)$. Consider a spectral window $I=[E_1,E_2]$ which
contains only regular values of the principal symbol $H$ of
$P$. Assume moreover that $H^{-1}(I+\interval{-\epsilon}{\epsilon})$
is compact in the phase space $M$, for some $\epsilon>0$. Then
$\sigma(P)\cap I$ is discrete, and is described as follows.

  \begin{enumerate} \item There exists an integer $d\geq 0$ such that
for each $E\in I$, the energy level set $H^{-1}(E)$ is the disjoint
union of $d$ smooth curves $\C_1(E),\dots, \C_d(E)$.  \item There
exist smooth functions $E\mapsto A_{k,j}(E)$, for $k=1,\dots,d$ and
$j\in\NM$, and smooth functions $I\times\interval[open left]01 \ni
(E,\h)\mapsto A_k(E;\h)$ admitting the asymptotic expansion \[
A_k(E;\h) \sim \frac{1}{\h}\sum_{j=0}^\infty \h^j A_{k,j}(E) \] in the
smooth topology, such that $\sigma(P)\cap I$ coincides modulo $\ohb$,
and including multiplicities, with the disjoint union
$\bigsqcup_{k=1}^d \sigma_k(\h)$, which we denote
by \begin{equation} \label{equ:bijection} \sigma(P)\cap I =
\bigsqcup_{k=1}^d \sigma_k(\h) + \ohb \end{equation}
with \begin{equation} \label{equ:sigmak} \sigma_k(\h) :=
A_k(\,\cdot\,;\h)^{-1}(2\pi \ZM) \,.  \end{equation} 
precisely, for each eigenvalue $E=E(\h)$ of $P$, for any 
the spectral multiplicity of the interval %
$B=\interval{E-\h^N}{E+\h^N}$ is, for $\h$ small enough, equal to %
the number of non-empty sets $\sigma_k \cap B$.  \item $A_{k,0}(E) =
\int_{\C_k(E)} \alpha$, where $\alpha$ is the Liouville 1-form, and
the map $E\to A_{k,0}(E)$ is a diffeomorphism from $I$ to its
image.  \item In particular, if $d=1$, then the spectrum in $I$ is
simple for $\h$ small enough, and gaps are of order $\O(\h)$.  \item
If $P$ has no subprincipal symbol, then $A_{k,1}$ is the Keller-Maslov
index of $\C_k$.  \end{enumerate} \end{theo}

\begin{rema}\label{rema:sloppy} We will be sloppy here on two
technical points: \begin{enumerate} \item The precise definition of
``two subsets of $\RM$ depending on $\h$ which coincide modulo $\ohb$
including multiplicity'': see~\cite[Definition
1.5]{san-fahs-letreust-raymond}. This notion is natural but some care
has to be taken at the ``boundary'' of the sets. See Step 4 of the
proof below.  \item The precise definition of the Keller-Maslov index,
both in the pseudodifferential~\cite{duistermaat-oscillatory} (see
also~\cite[Section 5.13]{guillemin-sternberg-semiclassical}) and the
Berezin-Toeplitz~\cite{charles-bs}
cases.  \end{enumerate} \end{rema} \begin{demo} \paragraph{Step
1. Discrete spectrum.} This step is only necessary in the
pseudodifferential case, where the phase space $M$ is not compact, and
it is now standard.  Let a smooth function $f$ be equal to 1 on $I$
and compactly supported in $I+\interval{-\epsilon}{\epsilon}$. Then
$1_I f = 1_I$ and hence by functional calculus $1_I(P)= 1_I(P)
f(P)$. The symbol $f(H(x,\xi))$ and all its derivatives vanish outside
the compact set $H^{-1}(I+\interval{-\epsilon}{\epsilon})$, so by
pseudodifferential functional calculus, $f(P)$ is compact
(see~\cite[Theorem 8.7, Theorem 9.4]{dimassi-sjostrand}, and
also~\cite[Section 13.6]{guillemin-sternberg-semiclassical}).  Thus
$1_I(P)$ is compact as well, which implies that its range is finite
dimensional.

  \paragraph{Step 2. Quasimodes.} Let us fix $k\in \{1,\dots, d\}$. We
  define the map $A_k$ to be the semiclassical action of
  Proposition~\ref{prop:semiclassical-action}.

  Thus, the map $E\to A_{k,0}(E)$ is the action integral on the energy
  set $H=E$; a classical computation shows that its derivative is the
  map $E\mapsto \tau_k(E)$ where $\tau_k(E)$ is the period of the
  Hamiltonian flow of $H$ on $\C_k(E)$. Since $E$ is a regular value,
  the Hamiltonian vector field cannot vanish on the level set, and
  hence $\tau_k(E)\neq 0$. This shows that $A_{k,0}$ is a
  diffeomorphism, and hence that the map $\h A_k(\,\cdot\,;\h)$ is
  invertible for $\h$ small enough.  This also implies that the
  inverse $G_k(\h) := [\h A_k(\,\cdot\,;\h)]^{-1}$ admits an
  asymptotic expansion in powers of $\h$. Hence, for convenience, we
  may define this inverse modulo $\ohb$ for all $\h\in \interval[open
  left]01$ by a Borel summation, which we call $G_k(\h)$ again.

  It now follows from~\eqref{equ:sigmak}
  that \begin{equation} \label{equ:n} E(\h) \in \sigma_k(\h) \ssi
  E(\h) = G_k(\h)(2\pi\h n) \qquad \text{for some }
  n\in\ZM \end{equation} (the integer $n$ being of course subjected to
  the fact that $2\pi n$ must belong to the range of $A_k(\,\cdot\,,
  \h)$.)

  For such a family $\{E(\h); \h\in J\}$, where
  $J\subset\interval[open left]01$ accumulates at zero, we may apply
  Proposition~\ref{prop:bs} and obtain a microlocal solution to
  $(P-E)\Psi_\h=\ohb$ in a neighborhood of $\C_k$. By a microlocal
  cutoff outside of $\C_k$ and vanishing on the other $\C_{k'}$'s,
  $k'\neq k$, we obtain a normalized quasimode $\Psi_\h^{(k)}$: \[
  \norm{(P-E)\Psi_\h^{(k)}}_{\mathcal{H}} = \ohb\,.  \] By the
  spectral theorem, there must be an element $\lambda(\h)\in
  \sigma(P)$ such that $\abs{\lambda(\h) - E(\h)} = \ohb$. This
  shows \[ \sigma_k \subset \sigma(P) + \ohb\,.  \]

  \paragraph{Step 3. Microsupport.}

  Let $\psi=\psi_\h\in\mathcal{H}$ be an eigenfunction of $P$ for the
  eigenvalue $E=E(\h)\in I$. Assume that $E(\h)\to E_0$, at least for
  some subsequence of values of $\h$. By semiclassical ellipticity,
  the microsupport of $\psi$ is contained in $H^{-1}(E_0)$, and is not
  empty if we normalize $\norm{\psi}=1$. This microsupport is
  invariant by the Hamiltonian flow of $H$ (this follows for instance
  from Proposition~\ref{prop:dim1}). Hence it is a finite union of
  curves $\C_k(E_0)$. Let $W_k$ be a saturated neighborhood of
  $\C_k(E_0)$: for $E$ close to $E_0$, the whole curve $\C_{E}$ is
  contained in $W_k$.

  Now, the microlocal restriction of $\psi$ on $W_k$ is a section of
  the sheaf $\mathcal{D}$ for all $E$ near $E_0$, and hence by
  Proposition~\ref{prop:bs} the corresponding Bohr-Sommerfeld
  conditions must be satisfied: thus $E\in \sigma_k(\h)$. This
  shows \[ \sigma(P) \subset \sigma_k + \ohb\,.  \]

\paragraph{Step 4. Multiplicities.}  For simplicity, we did not state
explicitly the result about multiplicities within the theorem, see
Remark~\ref{rema:sloppy}. Here is a more precise statement, which we
could not find in the literature (although it is not very far from the
method employed in~\cite{helffer-sjostrand}), except very recently in
the Berezin-Toeplitz case~\cite{deleporte-lefloch25}. While the idea
is simple (the multiplicity of $E$ should be exactly the number of $k$
for which $E\in \sigma_k$), the correct statement is more involved due
to boundary effects (eigenvalues may surreptitiously enter or leave
any given set when $\h$ varies arbitrarily little, and this may not be
detected because of the $\ohb$ error).

Let us prove the following:

\begin{quote} For each eigenvalue $E=E(\h)\in I$ of $P$, let $N\geq
2$, let $B=B(\h)$ be the ball around $E(\h)$, of radius $\h^N$.  Let
$N'>N$ and let $\tilde B$ be a slightly smaller ball, or radius $\h^N
- \h^{N'}$. Then, when $\h$ is small enough, \begin{enumerate} \item
the spectral multiplicity of $B$ is at least equal to the number of
$k$'s such that $\sigma_k$ intersects $\tilde B$; \item the spectral
multiplicity of $\tilde B$ is at most equal to the number of $k$'s
such that $\sigma_k$ intersects $B$.  \end{enumerate} \end{quote}

Here, we used the terminology ``spectral multiplicity of B'' to denote
the rank of the spectral projector of $P$ onto $B$.  First, remark
that in each $\sigma_k$, the discussion is simple: if two energies
$E_1(\h)$ and $E_2(\h)$ are separated by some $Ch^N$, $N\geq 1$, then
they must correspond to different integers $n_1$ and $n_2$
in~\eqref{equ:n}: hence $\abs{n_1(\h) - n_2(\h)}\geq 1$, and hence \[
\abs{E_1(\h) - E_2(\h) } \geq c \h \] for some $c>0$.

Let $B=B(\h)$ be a ball around some energy $E(\h)\in I$, of radius
$\h^N$ with $N\geq 2$, and let $\tilde B$ be a slightly smaller ball
as in the statement, so that $\tilde B + \ohb \subset B$. By the above
argument, for $\h$ small enough, each set $\sigma_k+\ohb$ can contain
at most one element of $\tilde B$. If this happens for two different
integers $k\neq\ell$, then the corresponding WKB quasimodes are
microsupported on different curves and hence their scalar product is
$\ohb$. Therefore they cannot be collinear for $\h$ small enough,
which means, using the variational characterisation of the spectrum,
that the spectral multiplicity of $B$ is at least equal to the number
of $k$'s such that $\sigma_k$ intersects $\tilde B$.

Conversely, as in the previous step, any eigenfunction in the range of
$1_{\tilde B}(P)$ must be supported on some union of $\C_k$'s and
hence is microlocally equal to a linear combination of the
corresponding WKB quasimodes. Since we have an orthonormal basis of
eigenfunctions, and since, again, the WKB quasimodes are mutually
almost orthogonal, we deduce that the spectral multiplicity cannot
exceed this number of $k$'s.

\end{demo}

We see from the proof that one actually has a precise description of
eigenfunctions and quasimodes.

\begin{theo}[Bohr-Sommerfeld eigenfunctions] With the same hypothesis
as Theorem~\ref{theo:bs}, \begin{enumerate} \item Any quasimode of $P$
(and hence, any eigenfunction) is equal, modulo $\ohb$, to a linear
combination of the WKB solutions.  \item Each of these WKB solutions,
with approximate eigenvalue $E=E(\h)$, must correspond to a component
$\C_k$ such that $\sigma_k$ intersects an $\ohb$ ball around
$E$.  \item For each $k$, to any solution $E(\h)$ of the
Bohr-Sommerfeld condition~\eqref{equ:bs} on the curve $\C_k$, one can
associate (for $\h$ small enough) a unique eigenvalue of $P$, for
which the corresponding WKB Ansatz is a
quasimode.  \end{enumerate} \end{theo}

\section{Caustics?}

Comparing with most texts on WKB (quasi)modes,
like~\cite{Hall2013-wkb}, it is perhaps surprising that no discussion
of turning points or introduction of the Airy function has been
necessary here. Actually, in our treatment, caustics have been
superbly ignored. What has become of them?

As we saw in the preceding sections, WKB solutions live intrinsically
on the curves $\C_k$, which are smooth lagrangian submanifolds of the
phase space $M$. In the usual pseudodifferential setting, a problem
occurs due to the fact that we ultimately need eigenfunctions that
depend on the position variable only, $x$; and in general the curves
$\C_k$ cannot be globally parameterized by $x$.

\begin{defi} Let $\Lambda\subset\RM^{2n}$ be a lagrangian manifold. A
point $(x,\xi)\in\Lambda$ is called \emph{caustic} if the map
$(x,\xi)\to x$ restricted to $\Lambda$ is \emph{not} a local
diffeomorphism.  \end{defi}

Thus, if $n=1$, any closed smooth curve $\C$ in $\RM^2$ must have at
least two caustics, corresponding to the extrema of $x$ on the
curve. The \emph{Maslov index} of $\C$ is the intersection number of
the curve with the vertical fibration: \emph{i.e.}, the number of
times the tangent to the curve crosses the vertical direction, counted
algebraically: $+1$ if the oriented curve stays locally on the left of
the vertical tangent. For a circle in $\RM^2$ oriented
counter-clockwise, the Maslov index is 2.  By homotopy invariance, the
Maslov index of any smooth loop in $\RM^2$ is $\pm 2$.

\begin{rema} In the Berezin-Toeplitz category, the caustic issue does
not exist, because elements of the Hilbert space $\mathcal{H}$ are
directly defined on the phase space $M$. As a matter of fact, we could
have avoided caustics altogether by transforming pseudodifferential
operators into Berezin-Toeplitz operators (in the non-compact phase
space $\CM^n$) using the FBI transform. This is not our method here,
but it can be done efficiently,
see~\cite{rouby-17,duraffour2025-bs}.  \end{rema} Symplectically, due
to the Darboux-Carathéodory theorem, any regular point of $H^{-1}(E)$
plays the same role, be it in $\RM^{2n}$ or in a compact Kähler
manifold, be it caustic or not. Since our method is to quantize the
Darboux-Carathéodory charts, the microlocal description completely
flies over the caustics. But, looking carefully, they are still there:
they influence the dimension of the $\theta$ variable in the
generalized WKB Ansatz~\eqref{equ:lagrangian} used for the kernel of
the Fourier Integral Operator quantizing the Darboux-Carathéodory
chart. By stationary phase, this dimension shows up in the asymptotic
expansion of the transition constant between two WKB solutions, which
is our Bohr-Sommerfeld cocycle~\eqref{equ:Cj-das}, see the proof of
Proposition~\ref{prop:semiclassical-action}. Then, in the final
result, Theorem~\ref{theo:bs}, the only remaining footprint of the
caustic analysis is the Maslov index, causing the famous
$\frac{1}{2}$-shift in the quantum numbers.

\section{Applications}

The description of Theorem~\ref{theo:bs} is so precise that almost any
asymptotic statement about eigenvalues of $P$ in
$I=\interval{E_1}{E_2}$ should directly follow from it. Here are some
examples, certainly known to experts, but which we could not find
explicitly in the literature. The hypothesis are the same as those of
the theorem.

\begin{coro}[Density of the spectrum] Given any $E_0\in I$, there
exists a family of eigenvalues $E(\h)\in \sigma(P)$ such that
$E(\h)\to E_0$, and more precisely $ E(\h) - E_0 =
\O(\h)$.  \end{coro} \begin{demo} Let us apply
Formula~\eqref{equ:sigmak}. Let $c(\h) = \h A_k (E_0; \h)$. We have \[
c(\h) = c_0 + \O(\h)\qquad c_0 = A_{k,0}(E_0)\,, \] and
$G_k(\h)(c(\h)) = E_0 + \ohb$. Let $n(\h) = \lfloor
\frac{c(\h)}{2\pi\h} \rfloor \in \ZM$, so that \[ \abs{2\pi n(\h)\h -
c(\h)} \leq 2\pi\h.  \] Since $G_k(\h)$ is locally Lipschitz,
uniformly in $\h$, we obtain \[ G_k(\h)(2\pi n(\h)\h) - G_k(\h)(c(\h))
= \O(\h) \] and by~\eqref{equ:sigmak} we obtain $E(\h)\in\sigma(P)$
with \[ G_k(\h)(2\pi n(\h)\h) = E(\h) + \ohb, \] which gives the
result.  \end{demo}

In the above result, we notice that the quantum number $n(\h)$ which
labels the eigenvalues is allowed to depend on $\h$, and the distance
$\abs{E(\h)-E_0}$ cannot in general be better than $\O(\h)$. A
different question is to try to investigate the behaviour, as $\h$
varies, of an eigenvalue with a given, fixed, $n$. This is what
concerns the following statement.

\begin{coro}[$\h$-behaviour of individual
eigenvalues]\label{coro:drift} All eigenvalues of $P$ in $I$ «depend
smoothly on $\h$ modulo $\ohb$» in the following sense: for any
$\varepsilon>0$ and any $N\geq 2$, there exists $\h_1>0$ such that the
following holds.  Let $\h_0\in \interval[open left]0{\h_1}$, and
consider an eigenvalue $E(\h_0)\in \sigma(P)\cap
\interval{E_1}{E_2}$. There exist a family of eigenvalues
$\{E(\h),\,\h \in \interval[open left]01\}$, and a smooth map $\h
\mapsto \lambda(\h)$ such that, as long as $\lambda(\h)\in I$, \[
\abs{ E(\h) - \lambda(\h)} \leq \varepsilon \h^N \,.  \] In the
pseudodifferential case, this family always leaves the interval $I$
for $\h$ small enough.  \end{coro} \begin{demo} We use the bijection
given by~\eqref{equ:bijection}. This defines, for $\h=\h_0$, an
integer $k\in\{1,\dots,d\}$, labelling a connected component of the
energy level set, and another integer $n\in\ZM$ (the quantum number)
labelling the approximate eigenvalue $A_k(\,\cdot\,; \h_0)^{-1}(2\pi
n)$. We let $\lambda(\h) := G_k(\h)(2\pi\h n)$ as in~\eqref{equ:n},
and the result follows. In other words, we just
reformulate~\eqref{equ:bijection} as \[ \sigma(P)\cap I =
\hspace{-3em}\bigsqcup_{k=1,\dots,d \atop n\in \ZM,\ 2\pi\h n\in
\textup{Dom}(G_k(\h))}\hspace{-3em}G_k(\h)(2\pi\h n) + \ohb \] which
expresses the spectrum as a union of $\h$-smooth branches indexed by
$(k,n)$.

  Note that $G_k(\h) = G_{k,0} + \O(\h)$, and $G_{k,0}$ is the
reciprocal of the action variable $E\mapsto A_{k,0}(E)$. In the
pseudodifferential case, the action is the area below $\C_k$ and hence
cannot vanish unless $E$ is a critical value. This means that
$G_{k,0}(0)$ cannot belong to $I$, which forces the eigenvalue to exit
this interval.  \end{demo}

Note that the \emph{drifting} of eigenvalues, forcing them to exit any
interval of regular values, was described for general quantum
integrable systems in~\cite{san-dauge-hall-rotation}.

From this corollary, all eigenvalues belong to smooth branches. In
general, the branches associated to different connected components may
intersect. Whether the exact eigenvalues follow the branches as $\h$
moves continuously (in the pseudodifferential setting), or instead
``hop'' to another branch, creating \emph{avoiding crossing} seems to
be an open question for general pseudodifferential operators. Of
course, for 1D Schrödinger operators, because of the simplicity of the
spectrum, we know that all crossings must be avoided.

\begin{coro}[Poor man's tunnel effect] Assume that two connected
components $\C_k$, $\C_\ell$ are ``quantum symmetric'' (for instance
they are exchanged by an affine symplectic map and $P$ is invariant
under the corresponding metaplectic transformation, or $P$ is a
Schrödinger operator with symmetric potential wells). Then we have
doublets of eigenvalues $E_k(\h)$ and $E_\ell(\h)$ such
that \begin{enumerate} \item $E_k(\h) = E_\ell(\h) + \ohb$; \item the
spectral multiplicity of any ball of radius $\h^N$, $N\geq 2$ around
$E_k(\h)$ (or $E_\ell(\h)$) is at least 2. It is exactly 2 if
$d=2$.  \end{enumerate} \end{coro} \begin{demo} The symmetry ensures
that the Bohr-Sommerfeld cocycles are the same for both connected
components. The result then follows directly from
Theorem~\ref{theo:bs}.  \end{demo}

See~\cite{duraffour-phd} for a thorough discussion of tunnelling for
1D pseudodifferential operators in relation to the Bohr-Sommerfeld
rules.

\medskip

Let us now count eigenvalues. The usual Weyl law (see for
instance~\cite[Theorem 10.1]{dimassi-sjostrand}) gives the number of
eigenvalues (with multiplicities) in $I$: \[ \mathcal{N}(P,I;\h) =
\frac{1}{2\pi\h} \textup{Vol}(H^{-1}(I)) + \O(1)\, \] where
$\textup{Vol}$ is the symplectic volume. So $\textup{Vol}(H^{-1}(I)) =
\int_{H^{-1}(I)} \abs{\omega}$ which by Stokes is (recall that
$\dd{\alpha}=\omega$) \begin{align*} \abs{\int_{H(x,\xi)=E_2} \alpha -
\int_{H(x,\xi)=E_2} \alpha} & = \sum_{k=1}^d
\abs{\oint_{\C_k(E_2)}\alpha - \oint_{\C_k(E_1)} \alpha} \\ & =
\sum_{k=1}^d \abs{A_{k,0}(E_2) - A_{k,0}(E_1)}\,.  \end{align*} Using
Theorem~\ref{theo:bs} we can not only go one step further in the
approximation, but also, surprisingly (maybe), obtain an exact
expression, without remainder.

\begin{coro}[Exact Weyl law] Assume $P$ has no subprincipal symbol.
Let $\tilde E_j(\h)\in I$, where $\h\in\interval[open left]01$ or $\h$
belongs to some subset accumulating at 0, be such
that \begin{enumerate} \item $\tilde{E}_j(\h)\to E_j$ as $h\to
0$; \item $\textup{dist}(\tilde E_j, \sigma(P))\geq \epsilon\h$ for
some $\epsilon>0$.  \end{enumerate} Then, for $\h$ small enough, the
number of eigenvalues of $P$ in $\tilde I := \interval{\tilde
E_1(\h)}{\tilde E_2(\h)}$ is exactly \begin{equation} \label{equ:weyl}
\mathcal{N}(P,\tilde I;\h) = \sum_{k=1}^d \lfloor \tfrac{1}{2\pi}
A_{k,0}(\tilde E_2) + \tfrac{1}{2} \rfloor - \lfloor \tfrac{1}{2\pi}
A_{k,0}(\tilde E_1) + \tfrac{1}{2} \rfloor \,.  \end{equation}
Moreover, for the fixed interval $I$, this gives \[
\mathcal{N}(P,I;\h) = \frac{1}{2\pi\h} \textup{Vol}(H^{-1}(I)) +
\frac{1}{2\pi}\big(\tau_k(E_2) - \tau_k(E_1)\big) + \delta(\h) +
\O(\h)\,, \] where $\tau_k(E)$ is the period of the Hamiltonian flow
of $H$ on $\C_k(E)$, and
$\abs{\delta(\h)}<1$.  \end{coro} \begin{demo} We use the
decomposition~\eqref{equ:bijection} and count the number of elements
$N_k(\h)$ of each set $\sigma_k(\h)$.

  By~\eqref{equ:sigmak}, $N_k(\h) = \#{(A_k(\,\cdot\,;\h)^{-1}(2\pi
\ZM)) }$; in other words, it is the number of integers $n\in\ZM$ such
that $A_k(E;\h)=2\pi n$ for some $E\in I$. Since $A_{k,0}$ is a local
diffeomorphism, $A_k$ is, for $\h$ small enough, a monotonous
bijection from $I$ into its image, so $N_k(\h)$ is the number of
integers in the interval $\interval{\beta_1}{\beta_2}$, where
$\beta_j:=\beta_j(\h) = A_k(\tilde E_j(\h);\h)/2\pi$, $j=1,2$.  By
assumption, $\beta_j$'s stay away from integers by a distance at least
$\tilde\epsilon$, for some $\tilde\epsilon>0$ depending on
$\epsilon$. So they cannot be integers, and \[ N_k(\h) = \lfloor
\beta_2(\h) \rfloor - \lfloor\beta_1(\h) \rfloor \,.  \] We have \[
\beta_j(\h) = \frac{1}{2\pi\h}A_{k,0}(\tilde E_j) +
\frac{1}{2\pi}A_{k,1}(\tilde E_j) + \O(\h) \] and if $P$ has no
subprincipal symbol, we have $A_{k,1}(E) = \pi$. What's more, when
$\h$ is small enough so that the remainder $\O(\h)$ is less than
$\frac{\tilde \epsilon}{2}$, the sum $\frac{1}{2\pi\h}A_{k,0}(E_j) +
\frac{1}{2\pi}A_{k,1}(E_j)$ cannot be an integer either. This gives \[
N_k(\h) = \lfloor \tfrac{1}{2\pi\h} A_{k,0}(\tilde E_2) + \tfrac{1}{2}
\rfloor - \lfloor \tfrac{1}{2\pi\h} A_{k,0}(\tilde E_1) + \tfrac{1}{2}
\rfloor\,.  \] Summing over $k$, we obtain the correct number modulo
$\ohb$. But since we are after an integer number, we can disregard the
$\ohb$ remainder for $\h$ small enough, which proves~\eqref{equ:weyl}.

Finally, we can write \[ N_k(\h) = \tfrac{1}{2\pi\h} A_{k,0}(\tilde
E_2) - \tfrac{1}{2\pi\h} A_{k,0}(\tilde E_1) + \delta \] where
$\abs{\delta}<1$.  Choosing $\tilde E_j$ such that $E_j - \tilde E_j =
\O(\h)$, and Taylor expanding at order 2, we obtain: \begin{equation}
N_k(\h) = \tfrac{1}{2\pi\h}( A_{k,0}( E_2) - A_{k,0}(E_1)) +
\tfrac{1}{2\pi}(\tau_k(E_2) - \tau_k(E_1)) + \delta +
\O(\h) \end{equation} which yields the result.  \end{demo}

Of course, in the statement of this corollary, we may replace the
interval $I$ by any closed subinterval.

The assumption on the vanishing subprincipal symbol was not crucial;
we obtain a general formula by using the correct value for $A_{k,1}$,
which is the sum of the Maslov index and the integral over $\C_k$ of
the subprincipal symbol, see~\cite{san-focus}.

It also follows from Theorem~\ref{theo:bs} that the moving endpoints
$\tilde E_j(\h)$ in the statement of this corollary always
exist. Actually, in the pseudodifferential case, one can even choose
$\tilde E_j(\h) = E_j$, provided we select a subsequence of $\h$'s;
this is due to the drifting phenomenon of Corollary~\ref{coro:drift}.

\paragraph{Inverse problems.} Theorem~\ref{theo:bs} is also key in
solving inverse problems. Let me just mention the Schrödinger case,
where we wish to recover the potential $V$ from the
spectrum~\cite{colin-inverse}, and the general pseudodifferential or
Berezin-Toeplitz case, where we recover the principal symbol modulo
symplectomorphisms~\cite{san-inverse,lefloch-phd}. In both case the
Bohr-Sommerfeld rules are crucial.

\section{Beyond EBK}

What's nice about the sheaf approach that we have presented here is
the ease of generalization. Let us end this paper by exploring some
extensions to the WKB/EKB methods.

\paragraph{Smaller errors.}  While an error term of size $\ohb$ is
extremely accurate, it is still insufficient for some important
applications like quantum tunnel effects.  A better accuracy, namely
$\O(e^{-\epsilon/\h})$, can be reached under analyticity
assumptions. One needs for this Sjöstrand's microlocal
theory~\cite{sjostrand-asterisque}, which was recently proven to
extend to Berezin-Toeplitz
operators~\cite{san-rouby-sjostrand20,deleporte-analytic21,charles2019analytic,deleporte-hitrik-sjoestrand24}. The
application to exponentially precise Bohr-Sommerfeld rules was
recently obtained by Duraffour~\cite{duraffour2025-bs}, based on ideas
from~\cite{hs-harper1}, in the pseudodifferential case, and
Deleporte-Le Floch~\cite{deleporte-lefloch25} in the Berezin-Toeplitz
case.

\paragraph{Integrable systems.}  As was already noted by
Brillouin~\cite{brillouin1926}, the equations for the semiclassical
action $e^{\frac{i}{\h}S_\h(x)}$ can be solved by quadrature in the 1D
case, or if the variables are separated. In fact, one can go even
further. A quantum completely integrable systems, with a phase space
of dimension $2n$, is a set of $n$ pairwise commuting self-adjoint
semiclassical operators $(P_1,\dots, P_n)$. The corresponding
classical symbols $(H_1,\dots, H_n)$ form a Liouville integrable
system. Since the Darboux-Carathéodory theorem extends to such
integrable systems, it can be shown that the whole microlocal strategy
holds, see~\cite{san-focus,san-fn}. This establishes both WKB
solutions which are joint quasimodes of the system, and
Bohr-Sommerfeld rules for the joint spectrum. To my knowledge, the
first mathematical treatment of these joint Bohr-Sommerfeld rules can
be found in~\cite{colinII,charbonnel,anne-charbonnel}. A review of
these results is presented in~\cite{san-panoramas}.

\paragraph{Singularities.}  While the Darboux-Carathéodory theorem
applies only to \emph{regular} points, the microlocal sheaf strategy
can, surprisingly, be extended to energy level sets containing
\emph{singular points}. For instance, a local extremum of the
potential $V$. The easiest case concerns elliptic singularities, for
which the Hamiltonian is a microlocal perturbation of the Harmonic
oscillator;
see~\cite{helffer-robert,colinI,lefloch-elliptic,deleporte_wkb_2023}.
The most spectacular application concerns hyperbolic points (for
instance, a local maximum of $V$). It has been worked out by Colin de
Verdière-Parisse~\cite{colin-p3} for the pseudodifferential case, and
Le Floch~\cite{lefloch-hyperbolic} in the Berezin-Toeplitz case. The
main idea is to extend Proposition~\ref{prop:dim1} to a neighborhood
of the critical point. One can show that, this time, the space of
microlocal solutions has dimension not one but \emph{two}. Of course,
the level sets are not circles anymore, but immersed curves with
transversal intersections. The number of branches (4) at each crossing
fits nicely with the dimension 2 of the microlocal solutions, and
produces new \emph{singular} Bohr-Sommerfeld conditions by expressing
the vanishing of some determinant related to the topology of the
critical fiber~\cite{colin-p3}. The analogous question for degenerate
singularities, for instance a quartic oscillator $\xi^2+x^4 +
\O(x^5)$, is still quite open;
see~\cite{colin-singularites,san-nikolay23}.

\paragraph{Non-selfadjoint operators.}  Given the heavy use of real
symplectic geometry that the microlocal method requires, it seems
hazardous to extend it to non-selfadjoint operators, \emph{i.e.}
semiclassical operators whose symbols are not necessarily
real-valued. Nevertheless, it turned out that this can be achieved,
and symplectic normal forms in the complexified phase space
remain very effective. For small non-selfadjoint perturbations of
self-adjoint pseudodifferential operators, Bohr-Sommerfeld rules with
$\ohb$ remainder have been obtained by Rouby~\cite{rouby-17},
following the pioneer works of
Melin-Sjöstrand~\cite{melin-sjostrand-non} and
Hitrik~\cite{Hitrik04}. The extension to more general 1D operators is
currently an active area of research, see for
instance~\cite{hitrik-zworski25,deleporte-lefloch25,san-bonthonneau-duraffour,reguer25}.

\bigskip

\paragraph{Acknowledgements.}
I would like to warmly thank Antide Duraffour and the anonymous
referee for their positive feedback and their valuable remarks and
corrections.

\bibliographystyle{abbrv}%
\bibliography{bibli-utf8}
\end{document}